\documentclass[a4paper, 12pt]{amsart}
\usepackage{amsmath, amssymb, eucal, amscd, amstext, enumerate}
\usepackage[all]{xy}
\usepackage{graphicx}

\newtheorem{theorem}{Theorem}[section]

\newtheorem{example}[theorem]{Example}
\newtheorem{proposition}[theorem]{Proposition}
\newtheorem{proposition and definition}[theorem]{Proposition and Definition}
\newtheorem{corollary}[theorem]{Corollary}

\newtheorem{remark}[theorem]{Remark}

\numberwithin{equation}{section}

\newcommand{\id}{{\rm id}}

\newcommand{\CA}{\mathcal{A}}

\newcommand{\BN}{\mathbb{N}}

\newcommand{\Id}{{\rm id}}

\newcommand{\CS}{\mathcal{S}}

\newcommand{\CT}{{\mathcal{T}}}

\begin{document}

\title[inductive limit]
{A construction of inductive limit \\for operator system}

\author{Jianze Li}

\address[Jianze Li]{Department of Mathematics, Tianjin University, Tianjin 300072, China}
\email{lijianze@tju.edu.cn}

\date{\today}


\thanks{This work was supported by the National
Natural Science Foundation of China (No.11601371).}


\begin{abstract}
In this paper,
we show a construction of inductive limit for operator system based on Archimedeanization.
This inductive limit may be not a closed operator system.
We prove that many nuclearity properties could be preserved by a special case of this inductive limit.
\end{abstract}

\maketitle

\section{Introduction}

Operator system has played an important role in $C^{\ast}$-algebra tensor product theory (see \cite{Effros-Ruan1,Murphy1,Paulsen1})
after its abstract characterization (see \cite{Choi-Effros1}).
Tensor product of operator systems was systematically studied in \cite{Kavruk-Paulsen1}.
Furthermore,
nuclearity properties of operator system were studied in \cite{Han-Paulsen1,Kavruk1,Kavruk-Paulsen1,Kavruk-Paulsen2}
and were applied to study the well known Kirchberg conjecture
(see \cite{Kirchiberg2}).

\medskip
Let $\CS$ and $\CT$ be operator systems and $\CS(\CS,\CT)$ be the set of unital completely positive maps.
A map $\varphi:\CS \rightarrow\CT$ is called a complete order monomorphism if it is a complete order isomorphism onto its image.
Let $\CS\otimes \CT$ be their algebraic tensor product.
Denote by $\CS\otimes_{min}\CT$ the operator system structure on $\CS\otimes \CT$
determined by the cones
\begin{align*}
M_{n}(\CS\otimes_{min}\CT)^+:=&\{(u_{i,j})\in M_n(\CS\otimes \CT): ((\varphi\otimes\psi)(u_{i,j}))_{i,j}\geq0\\
&\text{ for any}\  k,m\in\BN\text{ and}\ \varphi\in\CS(\CS,M_k), \psi\in\CS(\CT,M_m)\}.
\end{align*}
Then $min$ is injective in the sense that
$\CS_1\otimes_{min}\CT_1\subseteq \CS_2\otimes_{min}\CT_2$
whenever $\CS_1\subseteq \CS_2$ and $\CT_1\subseteq \CT_2$.
Let
\begin{align*}
D_n^{max}:=\{\alpha(P\otimes Q)\alpha^*: P\in M_l(\CS)^+, Q\in M_m(\CT)^+,
\alpha\in M_{n,lm},\ l,m\in\BN\}
\end{align*}
for all $n\in\BN$.
Denote by $\CS\otimes_{max}\CT$ the operator system structure on $\CS\otimes \CT$
determined by the cones
\begin{align*}
M_{n}(\CS\otimes_{max}\CT)^+:=&\{u\in M_n(\CS\otimes \CT): \varepsilon(1_\CS\otimes 1_\CT)_n+u\in D_n^{max}\ \text{ for any}\ \varepsilon>0\}.
\end{align*}
This Archimedeanization is a critical step in the construction of $max$ structure
(see \cite[Definition 5.4]{Kavruk-Paulsen1} and \cite[Proposition 3.16]{Paulsen-Todorov1}).

\medskip
Furthermore,
many other operator system structures on $\CS\otimes\CT$ were introduced,
including $e$, $el$, $er$ and $c$ (see \cite{Kavruk-Paulsen1}).
$el$ is left injective in the sense that
$\CS_1\otimes_{el}\CT\subseteq \CS_2\otimes_{el}\CT$
whenever $\CS_1\subseteq \CS_2$.
Similarly,
$er$ is right injective and $e$ is injective.
Let $\alpha,\beta\in\{min,e,el,er,c,max\}$ be two operator system structures.
We say $\beta$ is greater than $\alpha$,
denoted by $\alpha\leq\beta$,
if the identity map
is completely positive from $\CS\otimes_{\beta}\CT$ to $\CS\otimes_{\alpha}\CT$.
These structures can be summarized as follows:
$$min \leq e \leq el, er \leq c \leq max.$$
$\CS$ is called \emph{$(\alpha,\beta)$-nuclear},
if $\alpha\leq\beta$ and $\CS\otimes_{\alpha}\CT=\CS\otimes_{\beta}\CT$ for any operator system $\CT$.
More properties of these structures can be refer to \cite{Kavruk1,Kavruk-Paulsen1}.

The definition of inductive limit for closed operator system was introduced in \cite{Kirchiberg1} and
was used in \cite{Luthra01} to study the nuclearity properties of operator system.
In this paper,
we study a construction of inductive limit for arbitrary operator system.
This inductive limit may be not a closed operator system.
In section 2,
we show the detailed construction of this inductive limit based on Archimedeanization.
We prove the universal properties of this construction.
In particular,
if an operator system is the union of a sequence of increasing operator subsystems,
then it is the inductive limit.
In section 3,
we prove that all the nuclearity properties in \cite{Kavruk-Paulsen1} can be preserved by this special case of inductive limit.
In section 4,
we present two examples.

\section{Inductive limit}

\medskip
Let $\{\CS_k:k\in\BN\}$ be a sequence of operator systems and
$\varphi_k\in\CS(\CS_k,\CS_{k+1})$ for all $k\in\BN$.
Then $\{\CS_k,\varphi_k: k\in\BN\}$ is called an inductive sequence.
Let $\Pi_{k\in\BN}\CS_k$ be the direct sum of vector spaces.
Denote by
$$S' = \{(x_k),\ \text{there exists}\ k_0\ \text{such that}\ x_{k+1}=\varphi_k(x_k)\ \text{for all}\ k>k_0\}\subseteq\Pi_{k\in\BN}\CS_k.$$
Then $S'$ is a $\ast$-vector space with the canonical involution map.
Let $S'_{h}$ be the set of hermitian elements in $S'$.
Define
\begin{align*}
D_1 :=&\{(x_k)\in S'_h:\text{there exists $k_0$ such that $x_{k+1}=\varphi_k(x_k)$ and $x_k\in \CS_k^+$ for all $k>k_0$}\},\\
D_n :=&\{((u_{i,j}^{(k)}))\in M_n(S')_h: \text{there exists $k_0$ such that}\\
&\text{$(u_{i,j}^{(k+1)})=(\varphi_k(u_{i,j}^{(k)}))$ and $(u_{i,j}^{(k)})\in M_n(\CS_k)^{+}$ for all $k>k_0$}\}\subseteq M_n(S')_h
\end{align*}
for all $n\in\BN$.
Let $K\subseteq S'$ be the set of $(x_k)$ such that there exists $k_0$ for which $x_{k}=0$ for all $k>k_0$.
Let $S:=S'/K$ be the quotient vector space with $q:S'\rightarrow S'/K$ as the canonical quotient map.
We define
\begin{align*}
&S^{+}:=\{u\in S_h:\text{there exists $a\in D_1$ such that $u=q(a)$}\},\\
&M_n(S)^+:=\{(u_{i,j})\in M_n(S)_h:\text{there exists $(a_{i,j})\in D_n$ such that $(u_{i,j})=q_n((a_{i,j}))$}\}
\end{align*}
for all $n\in\BN$.
Note that for any $(u_{i,j})\in M_n(S)_h$,
there exists $(b_{i,j})\in M_n(S')_h$ such that $(u_{i,j})=q_n((b_{i,j}))$.
Then $S$ is a matrix ordered $\ast$-vector space with $(1_{\CS_k})+K\in S$ as a matrix ordered unit.
Let $\CS = \text{Arch}(S)$ be the Archimedeanization of $S$.
Then $\CS$ is called the inductive limit of inductive sequence $\{\CS_k,\varphi_k: k\in\BN\}$,
denoted by $\CS=\lim\limits_{\longrightarrow}\CS_k$ or simply $\CS_k\rightarrow\CS$.

\medskip
Now we would like to show some notations before the universal properties.
For any $p,q\in\BN$ with $q>p$,
we denote by $\varphi_{p,q}$ the composition of $\{\varphi_{k}:p\leq k<q\}$.
For any $x\in\CS_k$, we denote by $\hat{x}$ the sequence $(x_k)$
with $x_l=0$ for $l<k$,
$x_k=x$ and $x_l=\varphi_{k,l}(x)$ for $l>k$.
Let $\varphi^{(k)}:\CS_k\rightarrow \CS$ be the map which sends $x\in\CS_k$ to $\hat{x}+K$.
Then $\varphi^{(k)}$ is unital completely positive and
the following diagram commutes for all $k\in\BN$.
$$\xymatrix{
\CS_k \ar[rr]^{\varphi_k}\ar[dr]_{\varphi^{(k)}} & & \CS_{k+1} \ar[dl]^{\varphi^{(k+1)}} \\
& \CS &
}$$

\begin{remark}\label{remark-1}
(i) It is clear that $\CS$ is the union of the increasing sequence of operator subsystems $\varphi^{(k)}(\CS_k)$,
that is
$\CS = \cup_{k\in\BN}\varphi^{(k)}(\CS_k).$
This is different from the inductive limit for closed operator system in \cite{Kirchiberg1,Luthra01},
which is the closure of union.\\
(ii) $\varphi^{(k)}(x)\in \CS^+$ if there exist $l>k$ such that $\varphi_{k,l}(x)\in\CS_{l}^+$.\\
(iii) For any $v\in M_n(S)^+$,
we can find $k_0\in\BN$ such that for any $k>k_0$,
there exists $u\in M_n(\CS_k)^+$ such that $\varphi^{(k)}_n(u)=v$.
\end{remark}

\begin{remark}\label{remark-inclusion}
Suppose that $\CS$ is an operator system and $\{\CS_k,k\in\BN\}$
is an increasing sequence of operator subsystems.
If $\CS_k\subseteq\CS_{k+1}$ with the inclusion map $i_k$ for $k\in\BN$
and $\CS = \cup_{k\in\BN}\CS_k$,
then  $\CS=\lim\limits_{\longrightarrow}\CS_k$.
For convenience,
we denote by $\CS_k\xrightarrow{i}\CS$ in this case.
\end{remark}

\begin{remark}
It is not hard to see that any separable operator system is the inductive limit
of an increasing sequence of finite-dimensional operator systems.
This shows the importance of finite-dimensional operator system
in operator system theory.
\end{remark}

\begin{proposition}\label{proposition-property}
Let $\CS$ be the inductive limit of inductive sequence $\{\CS_k,\varphi_k: k\in\BN\}$.\\
(i) Let $x\in\CS_k, y\in\CS_l$ with $k,l\in\BN$.
Then $\varphi^{(k)}(x)=\varphi^{(l)}(y)$ if and only if there exists $p>k,l$ such that $\varphi_{k,p}(x)=\varphi_{l,p}(y)$.\\
(ii) If $\CT$ is an operator system and there exist unital completely positive maps $\psi^{(k)}:\CS_k\rightarrow\CT$
such that the following left diagram commutes for all $k\in\BN$.
$$\xymatrix{
\CS_k \ar[rr]^{\varphi_k}\ar[dr]_{\psi^{(k)}} & & \CS_{k+1} \ar[dl]^{\psi^{(k+1)}}
&&& \CS_k \ar[rr]^{\varphi^{(k)}}\ar[dr]_{\psi^{(k)}} & & \CS \ar@{.>}[dl]^{\Psi} \\
& \CT & &&& & \CT &
}$$
Then there exists a unique unital completely positive map $\Psi:\CS\rightarrow\CT$
such that the above right diagram commutes for all $k\in\BN$.
\end{proposition}

\begin{proof}
We only prove (ii).
For $(x_k)+K\in\CS$,
the sequence $\{\psi^{(k)}(x_k): k\in\BN\}$
is a sequence in $\CT$ with constant element except finite positions.
Denote by $\lim\psi^{(k)}(x_k)$ the constant element.
It is clear that $\lim\psi^{(k)}(y_k)=\lim\psi^{(k)}(x_k)$,
if $(x_k)+K=(y_k)+K$.
Now we define $\Psi:\CS\rightarrow\CT$ sending $(x_k)+K\in\CS$ to $\lim\psi^{(k)}(x_k)$.
It is clear that the above right diagram commutes and $\Psi$ is unital.
If $(x_k)+K\in S^+$,
then $\lim\psi^{(k)}(x_k)$ is positive.
We can verify that $\Psi$ is completely positive similarly.
Note that $\CS$ is the Archimedeanizaiton of $S$.
We get that $\Psi$ is a unital completely positive map.
\end{proof}

\begin{proposition}
Let $\CS$ and $\CT$ be inductive limits of inductive sequences $\{\CS_k,\varphi_k: k\in\BN\}$ and $\{\CT_k,\psi_k: k\in\BN\}$,respectively.
Let $\varphi^{(k)}:\CS_k\rightarrow\CS$ and $\psi^{(k)}:\CT_k\rightarrow\CT$ be the canonical maps.
Suppose that there exist unital completely positive maps
$\pi_k:\CS_k\rightarrow\CT_{k}$ such that the following left diagram commutes for all $k\in\BN$.
$$\xymatrix{
\CS_k \ar[rr]^{\varphi_k}\ar[d]^{\pi_{k}} & & \CS_{k+1} \ar[d]^{\pi_{k+1}}
&&& \CS_k \ar[rr]^{\varphi^{(k)}}\ar[d]^{\pi_{k}} & & \CS \ar@{.>}[d]^{\pi} \\
\CT_k \ar[rr]^{\psi_k} & & \CT_{k+1}
&&& \CT_k \ar[rr]^{\psi^{(k)}} & & \CT
}$$
Then there exists a unique
$\pi\in\CS(\CS,\CT)$ such that the above right diagram commutes for all $k\in\BN$.
\end{proposition}

\begin{proof}
Denote by $T=T'/N$ such that
$\CT = \text{Arch}(T)$.
For any $(x_k)+K\in\CS$,
we have that $(\pi_k(x_k))+N\in\CT$.
It is clear that
$(\pi_k(x_k))+N=(\pi_k(y_k))+N$
if $(x_k)+K=(y_k)+K$.
Now we define $\pi:\CS\rightarrow\CT$ sending
$(x_k)+K$ to $(\pi_k(x_k))+N$.
Then the above right diagram commutes and $\pi$ is unique.
Since $\pi:S\rightarrow T$ is completely positive,
then $\pi:\CS\rightarrow\CT$ is also.
\end{proof}

\section{Nuclearity properties}

In this section,
we prove that all the nuclearity properties in \cite{Kavruk-Paulsen1} can be preserved
by the inductive limit in the case of Remark \ref{remark-inclusion}.
Some similar results were proved in \cite{Luthra01} based on different inductive limit definitions and different methods.

\begin{proposition}\label{proposition-max}
Let $\CT$ be an operator system and $\CS_k\xrightarrow{i}\CS$.
Then $\CS\otimes_{max}\CT$ is the inductive limit of
$\{\CS_k\otimes_{max}\CT,i_k\otimes id_{\CT}: k\in\BN\}$.
\end{proposition}

\begin{proof}
Note that $i^{(k)}\otimes\Id_{\CT}:\CS_k\otimes_{max}\CT\rightarrow\CS\otimes_{max}\CT$ is completely positive for any $k\in\BN$.
There exists a unique unital completely positive map $\Psi:\lim\limits_{\longrightarrow}(\CS_k\otimes_{max}\CT)\rightarrow\CS\otimes_{max}\CT$
such that the following diagram commutes.

$$\xymatrix{
& & \CS\otimes_{max}\CT  &
& & \frac{\Pi\CS_k}{K_0}\otimes\CT   \\
\CS_k\otimes_{max}\CT       \ar[rr]^{\Phi^{(k)}}     \ar[rru]^{i^{(k)}\otimes\Id_{\CT}}
& & \lim\limits_{\longrightarrow}(\CS_k\otimes_{max}\CT) \ar@{.>}[u]_{\Psi}   &
\CS_k\otimes\CT       \ar[rr]^{\Phi^{(k)}}     \ar[rru]^{i^{(k)}\otimes\Id_{\CT}}
& & \frac{\Pi(S_k\otimes\CT)}{K} \ar@{.>}[u]_{\Psi}
}$$

It is clear that $\Psi$ is surjective.
Now we show that $\Psi$ is injective.
Let $v\in\lim\limits_{\longrightarrow}(\CS_k\otimes_{max}\CT)$ with $\Psi(v)=0$.
There exists $u\in\CS_k\otimes_{max}\CT$ for large enough $k\in\BN$ such that $\Phi^{(k)}(u)=v$
and thus $(i^{(k)}\otimes\Id_{\CT})(u)=0$, that is u=0.
Therefore, v=0 and thus $\Psi$ is injective.

Now we prove that $\Psi$ is a complete order isomorphism.
Note that  $M_{n}(\CS\otimes_{max}\CT)^{+}$ is the Archimedeanization of
$$D_{n}^{max}=\{\alpha(P\otimes Q)\alpha^{\ast}:P\in M_{l}(\CS)^{+}, Q\in M_{m}(\CT)^{+}, \alpha\in M_{n,lm}, l,m\in\BN\}$$
and for any $P\in M_{l}(\CS)^{+}$,
there exists large enough $k\in\BN$ and $P'\in M_{l}(\CS_k)^{+}$ such that $i^{(k)}_{l}(P')=P$.
Then we get that
$$D_{n}^{max}\subseteq \bigcup_{k\in\BN}(i^{(k)}\otimes\Id_{\CT})_{n}(M_{n}(\CS_k\otimes_{max}\CT)^{+})$$
for any $n\in\BN$.
Since the diagram commutes,
we have
$D_{n}^{max}\subseteq \Psi_{n}(M_{n}(\lim(\CS_k\otimes_{max}\CT))^{+})$
for any $n\in\BN$.
Since $\Psi$ is completely positive,
we get that
$$M_{n}(\CS\otimes_{max}\CT)^{+}=\Psi_{n}(M_{n}(\lim(\CS_k\otimes_{max}\CT))^{+})$$
for any $n\in\BN$.
\end{proof}

\begin{corollary}\label{corollary-er-c}
Let $\CT$ be an operator system and $\CS_k\xrightarrow{i}\CS$.
Then \\
(i) $\CS\otimes_{er}\CT$ is the inductive limit of
$\{\CS_k\otimes_{er}\CT,i_k\otimes id_{\CT}: k\in\BN\}$.\\
(ii) $\CS\otimes_{c}\CT$ is the inductive limit of
$\{\CS_k\otimes_{c}\CT,i_k\otimes id_{\CT}: k\in\BN\}$.
\end{corollary}

\begin{proof}
The proof is based on Proposition \ref{proposition-max} and \cite[Theorem 6.4,Theorem 7.5]{Kavruk-Paulsen1}.
\end{proof}

\begin{remark}\label{remark-03}
Let $\CS$ be an operator system and $\CT_k\xrightarrow{i}\CT$.
By the similar methods in Proposition \ref{proposition-max} and Corollary \ref{corollary-er-c},
we get\\
(i) $\CS\otimes_{max}\CT$ is the inductive limit of
$\{\CS\otimes_{max}\CT_{k},id_{\CS}\otimes i_k: k\in\BN\}$.\\
(ii) $\CS\otimes_{c}\CT$ is the inductive limit of
$\{\CS\otimes_{c}\CT_k,id_{\CS}\otimes i_k: k\in\BN\}$.\\
(iii) $\CS\otimes_{el}\CT$ is the inductive limit of
$\{\CS\otimes_{el}\CT_k,id_{\CS}\otimes i_k: k\in\BN\}$.
\end{remark}

\begin{remark}\label{remark-04}
(i) Let $\alpha\in\{min,e,el\}$ and $\CS_k\xrightarrow{i}\CS$.
It is clear that $\CS\otimes_{\alpha}\CT$ is the inductive limit of
$\{\CS_k\otimes_{\alpha}\CT,i_k\otimes id_{\CT}:k\in\BN\}$,
since $\alpha$ is left injective.\\
(ii) Let $\alpha\in\{min,e,er\}$ and $\CT_k\xrightarrow{i}\CT$.
It is clear that $\CS\otimes_{\alpha}\CT$ is the inductive limit of
$\{\CS\otimes_{\alpha}\CT_k,id_{\CS}\otimes i_k:k\in\BN\}$,
since $\alpha$ is right injective.
\end{remark}

\begin{remark}\label{remark-02}
Let $\alpha,\beta\in\{min,e,el,er,c,max\}$ and $\alpha\leq\beta$.
Note that Remark \ref{remark-03} and \ref{remark-04}.
The following are equivalent.\\
(i) $\CS$ is $(\alpha,\beta)$-nuclear.\\
(ii) $\CS\otimes_{\beta}\CT_k\rightarrow\CS\otimes_{\alpha}\CT$
if $\CT_k\xrightarrow{i}\CT$.\\
(iii) $\CS\otimes_{\alpha}\CT_k\rightarrow\CS\otimes_{\beta}\CT$
if $\CT_k\xrightarrow{i}\CT$.
\end{remark}

\begin{theorem}\label{theorem-nuclear}
Let $\CS_k\xrightarrow{i}\CS$.
Suppose that $\alpha,\beta\in\{min,e,el,er,c,max\}$ and $\alpha\leq\beta$.\\
(i) $\CS$ is $(\alpha,\beta)$-nuclear if there exists $k_0$
such that $\CS_k$ is $(\alpha,\beta)$-nuclear for any $k>k_0$;
(ii) If $\beta=e$ or $el$,
then $\CS$ is $(\alpha,\beta)$-nuclear if and only if there exists $k_0$
such that $\CS_k$ is $(\alpha,\beta)$-nuclear for any $k>k_0$.
\end{theorem}
\begin{proof}
(i) If $\alpha\in\{min,e,el\}$,
then $\alpha$ is left injective and thus $\CS_k\otimes_{\alpha}\CT\subseteq\CS\otimes_{\alpha}\CT$ for any $k\in\BN$.
We only need to prove that $\id_{\CS\otimes\CT}$ in the following left diagram is completely positive.
For any $u\in (\CS\otimes_{\alpha}\CT)^+$,
there exists large enough k such that $u\in(\CS_k\otimes_{\alpha}\CT)^+=(\CS_k\otimes_{\beta}\CT)^+$.
Then $u\in (\CS\otimes_{\beta}\CT)^+$ and thus
$\id_{\CS\otimes\CT}$ is positive.
We can prove it is completely positive similarly.
If $\alpha\in\{er,c,max\}$,
by Proposition \ref{proposition-max} and Corollary \ref{corollary-er-c},
we get that $\CS\otimes_{\alpha}\CT = \CS\otimes_{\beta}\CT$.
(ii) is clear from the following right diagram since $\alpha$ and $\beta$ are both left injective.
$$\xymatrix{
\CS\otimes_{\alpha}\CT   \ar[rr]^{\id_{\CS\otimes\CT}}& & \CS\otimes_{\beta}\CT &&
\CS\otimes_{\alpha}\CT   \ar[rr]^{\id_{\CS\otimes\CT}}& & \CS\otimes_{\beta}\CT
\\
\CS_k\otimes_{\alpha}\CT  \ar@{=}[rr] \ar@{^{(}->}[u] & &  \CS_k\otimes_{\beta}\CT\ar[u]_{i^{(k)}\otimes\id_{\CT}} &&
\CS_k\otimes_{\alpha}\CT  \ar[rr]^{\id_{\CS_k\otimes\CT}} \ar@{^{(}->}[u] & &  \CS_k\otimes_{\beta}\CT\ar@{^{(}->}[u]
}$$
\end{proof}

\section{Example}

\begin{example}
Recall that a UHF algebra (uniformly hyperfinite algebra) is a unital $C^\ast$-algebra $\CA$
such that $\CA = \overline{\cup_{n=1}^{\infty}\CA_n}$,
where $\{\CA_n:n\in\BN\}$ is an increasing sequence of finite dimensional simple $C^\ast$-subalgebras
containing the unit of $\CA$ (see \cite{Murphy1}).
Let $n,d\in\BN$.
We denote by
$$\varphi: M_n \rightarrow M_{dn}, \ \   x\mapsto
\left(
\begin{array}{ccc}
x &   & 0 \\
  & \cdots &  \\
0 &   & x
\end{array}
\right) $$
the canonical map from $M_n$ to $M_{dn}$.
Note that any finite dimensional simple $C^\ast$-algebra is $\ast$-isomorphic to some $M_k$.
Then $\CA$ is in fact the $C^\ast$-algebra inductive limit of a sequence
$\{M_{n_k},\varphi_k\}$ with $n_k|n_{k+1}$ and $\varphi_k$ is the canonical map from $M_{n_k}$ to $M_{n_{k+1}}$.

Let $\gamma:\BN\setminus\{0\}\rightarrow\BN\setminus\{0\}$ be a function.
Define $\gamma!:\BN\setminus\{0\}\rightarrow\BN\setminus\{0\}$ by
$$\gamma!(n)=\gamma(1)\gamma(2)\dots\gamma(n).$$
Let $\varphi_n:M_{\gamma!(n)}\rightarrow M_{\gamma!(n+1)}$ be the canonical map.
Then any $\gamma$ determines a $C^\ast$-algebra inductive sequence $\{M_{\gamma!(n)},\varphi_n\}$
and thus a UHF algebra $\CA_\gamma$.

Now we consider the operator system inductive limit of this sequence.
Let $\CS_\gamma=\lim\limits_{\longrightarrow}\{M_{\gamma!(n)},\varphi_n\}$.
Then there exists $\sigma:\CS_\gamma\rightarrow \CA_\gamma$
which is a complete order monomorphism.
In fact,
if $\varphi^{(n)}:M_{\gamma!(n)}\rightarrow \CA_\gamma$ is the canonical $\ast$-homomorphism,
there exists a unique unital completely positive map $\sigma$ from $\CS_\gamma$ to $\CA_\gamma$
such that the following diagram commutes by Proposition \ref{proposition-property}.
$$\xymatrix{
& & \CA_\gamma
\\
M_{\gamma!(n)}      \ar[rr]^{\tau^{(k)}}     \ar[rru]^{\varphi^{(n)}}
& & \CS_\gamma \ar@{.>}[u]_{\sigma}
}$$
By the constructions of $C^\ast$-algebra inductive limit,
we see that $\CS_\gamma$ is in fact a linear subspace of $\CA_\gamma$ and $\sigma$ is the inclusion map.
Then it is injective.
Now we prove that $\sigma$ is a complete order monomorphism.
Let $l\in\BN$ and $\sigma_{l}(v)\in M_l^{+}(\sigma(\CS_\gamma))$.
Since $\sigma(\CS_\gamma)=\cup_{n\in\BN}\varphi^{(n)}(M_{\gamma!(n)})$,
there exists large enough $k\in\BN$ such that
$\sigma_{l}(v)\in M_l^{+}(\varphi^{(k)}(M_{\gamma!(k)}))$.
Then $v\in M_l^{+}(\CS_\gamma)$ since $\tau^{(k)}$ is the inclusion.
\end{example}

\begin{example}
Recall that
$\CS_{\infty}$
is the smallest closed operator subsystem in $C^{\ast}(\mathbb{F}^{\infty})$ containing all the generators (see \cite{Kavruk-Paulsen2}).
$\CS_n$ is the 2n+1 dimensional operator system in $C^{\ast}(\mathbb{F}^{n})$ containing the generators.
There exist unital completely positive maps
$i^{(n)}:\CS_n\rightarrow \CS_\infty$ and $p^{(n)}:\CS_\infty\rightarrow\CS_n$ such that
$p^{(n)}\circ i^{(n)}$ is the identity map on $\CS_n$ for any $n\in\BN$.
Then $i^{(n)}$ is a complete order monomorphism and thus $\CS_n\subseteq \CS_\infty$.
Similarly, $\CS_n\subseteq\CS_{n+1}$ in a canonical way.

Now we define
$$\CS_\varepsilon=\text{span}\{g_i^{\ast},e,g_i: i\in\BN\}\subseteq C^{\ast}(\mathbb{F}^{\infty}).$$
It is clear that $\CS_{\infty}=\bar{\CS}_\varepsilon$ and
$\CS_{\varepsilon}=\lim\limits_{\longrightarrow}\{\CS_n,i_n,n\in\BN\}$,
that is,
$\CS_{\varepsilon}=\cup_{n\in\BN}\CS_n$.
Then $\CS_\varepsilon$ is $(min,er)$-nuclear by \cite[Proposition 9.9]{Kavruk-Paulsen2} and Theorem \ref{theorem-nuclear}.
By the similar methods for \cite[Theorem 9.13]{Kavruk-Paulsen2},
we could get that Kirchberg conjecture has an affirmative answer if and only if $\CS_\varepsilon$ is $(el,c)$-nuclear.
\end{example}

\end{document}